\RequirePackage{fix-cm}
\documentclass{svjour3}                     
\smartqed  
\usepackage[parfill]{parskip}    
\usepackage{amsmath, amssymb, latexsym, enumerate, mathrsfs}
\usepackage{graphicx}
\usepackage{hyperref}
\usepackage{graphicx}
\usepackage{array}
\usepackage{times}
\usepackage{moreverb}
\usepackage{marvosym}
\def\qed{\hfill $\Box$}
\def\etal{\mbox{\em et al.}}
\usepackage{latexsym}
\begin{document}

\title{On the entropy of equilibrium measures and game-theoretic equilibrium feedback operators in multi-channel dynamical systems
}

\titlerunning{On the entropy of equilibrium measures and game-theoretic equilibrium feedback operators}        

\author{Getachew K. Befekadu}

\author{Getachew K. Befekadu \and Panos~J.~Antsaklis
}

\institute{G. K. Befekadu~ ({\large\Letter}\negthinspace) \at
	          Department of Electrical Engineering, University of Notre Dame, Notre Dame, IN 46556, USA. \\
	          Tel.: +1 574 631 6618\\
	          Fax: +1 574 631 4393 \\
	          \email{gbefekadu1@nd.edu}           
	          \and
	          P. J. Antsaklis \at
	          Department of Electrical Engineering, University of Notre Dame, Notre Dame, IN 46556, USA. \\
	          \email{antsaklis.1@nd.edu}           
}

\date{January 30, 2013}

\maketitle

\begin{abstract}
We investigate the connection between the entropy of equilibrium measures and game-theoretic equilibrium feedback operators in a multi-channel dynamical system. Specifically, we show that the existence of an equilibrium measure, which maximizes the free energy (i.e., the sum of the entropy and the integral over a potential), is related to an equilibrium or ``maximum entropy" state for the multi-channel dynamical system that is composed with a set of feedback operators. Further, we observe that such a connection makes sense when this set of feedback operators strategically interacts over an infinite-horizon, in a game-theoretic sense, using the current state-information of the system. Finally, we briefly comment on the implication of our result to the resilient behavior of the equilibrium feedback operators, when there is a random perturbation in the system.  
\end{abstract}

\section{Introduction} \label{S1}
The main purpose of this paper (which is a continuation of our previous paper \cite{BefAn13}) is to draw a connection between the entropy of equilibrium measures and game-theoretic equilibrium feedback operators in multi-channel  dynamical systems. We first specify a game in a strategic form over an infinite-horizon -- where, in the course of the game, each feedback operator generates a feedback control in response to the action of other feedback operators through the system (i.e., using the current state-information of the system) and, similarly, any number of feedback operators can decide on to play their feedback strategies simultaneously. However, each of these feedback operators are expected to respond in some sense of best-response correspondence to the strategies of the other feedback operators in the system. In such a scenario, it is well known that the notion of game-theoretic equilibrium will offer a suitable framework to study or characterize the robust property of all equilibrium solutions under a family of information structures -- since no one can improve his payoff by deviating unilaterally from this strategy once the equilibrium strategy is attained (e.g., see \cite{Aub93} or \cite{Nas51} on the notions of optimums and strategic equilibria in games).

In view of the above arguments, we present in this paper a game-theoretic formalism -- where the criterion is to maximize the free energy (i.e., the sum of the entropy and the integral over potential) for the multi-channel dynamical system, which is composed with a set of feedback operators. Such a formalism is an intuitive for establishing a connection between the existence of an equilibrium measure that is related to the equilibrium or ``maximum entropy" state for dynamical systems (e.g., see also \cite{Rue78}, \cite{Sin72} or \cite{Kel98} for a detailed exposition of equilibrium states for dynamical systems), and a set of feedback operators that strategically interacts over an infinite-horizon using the current state-information of the system. Finally, we also comment on the implication of our result to the resilient behavior of the equilibrium feedback operators, when there is a random perturbation in the system. Here, we hasten to add that such a study, which involves evidence of dynamical systems exhibiting resilient behavior, would undoubtedly provide a better understanding of reliability versus the degree of redundancy in large complex systems.\footnote{In this paper, our intent is to provide a theoretical framework, rather than pointing out a specific numerical or application problem.}

The paper is organized as follows: In Section~\ref{S2} we provide some preliminary results that are relevant to our paper. In Section~\ref{S3}, we present our main results -- where we establish a connection between the entropy of equilibrium measures and game-theoretic equilibrium feedback operators in multi-channel dynamical systems. This section also discusses an extension of the resilient behavior (to these equilibrium feedback operators), when there is a random perturbation in the system.

\section{Preliminaries} \label{S2}
\subsection{General setting} \label{S2(1)}
Consider the following multi-channel dynamical system  
\begin{align} 
 \dot{x}(t) &= A(t) x(t) + \sum\nolimits_{j \in \mathcal{N}} B_j(t) u_j(t), ~~ x(t_0)=x_0, \quad t \in [t_0, +\infty ),  \label{EQ1}
\end{align}
where $A(\cdot) \in \mathbb{R}^{d \times d}$, $B_j(\cdot)  \in \mathbb{R}^{d \times r_j}$, $x(t) \in X \subset \mathbb{R}^{d}$ is the state of the system, $u_j(t) \in U_j$ is the control input to the $j$th\,-\,channel and $\mathcal{N} \triangleq \{1, 2, \ldots, N\}$ represents the set of control input channels (or the set of feedback operators) in the system. 

Moreover, we consider the following class of admissible control strategies that will be useful in the following section (i.e., in Subsection~\ref{S3(1)} of game-theoretic formalism)
\begin{align}
  U_{\mathcal{K}} \subseteq \biggm\{u(t) \in \prod\nolimits_{j \in \mathcal{N}} \underbrace{L^{2}(\mathbb{R}_{+},\mathbb{R}^{r_j}) \cap L^{\infty}(\mathbb{R}_{+},\mathbb{R}^{r_j})}_{\triangleq U_j}\biggm\},  \label{EQ2}
\end{align}
where $u(t)$ is given by $u(t)=\bigl(u_1(t), \, u_{2}(t), \, \ldots, \, u_N(t)\bigr)$.

In what follows, suppose there exist feedback operators $\bigl(\mathcal{K}_1^{\ast},\, \mathcal{K}_2^{\ast},\, \ldots,\, \mathcal{K}_N^{\ast}\bigr)$ from a class of linear operators \,$\mathscr{K} \colon X \rightarrow  U_{\mathcal{K}}$ (i.e., $(\mathcal{K}_jx)(t) \in U_j$ for $j \in \mathcal{N}$) with strategies $\bigl(\mathcal{K}_j^{\ast}x\bigr)(t) \in U_j$ for $t \ge t_0$ and for $j \in \mathcal{N}$. Further, let $\phi_j\bigl(t; t_0, x_0, \bigl(\widehat{u_j(t), u^{\ast}_{\neg j}(t)}\bigr)\bigr) \in X$ be the unique solution of the $j$th\,-\,subsystem
\begin{align}
 \dot{x}^j(t) &= \Bigl(A(t)+ \sum\nolimits_{i \in \mathcal{N}_{\neg j}} B_i(t) \mathcal{K}_i^{\ast}(t) \Bigr) x^j(t) + B_j(t) u_j(t),  \label{EQ3}
\end{align}
with an initial condition $x_0 \in X$ and control inputs given by
\begin{align}
\bigl(\widehat{u_j(t), u^{\ast}_{\neg j}(t)}\bigr)\triangleq\bigl(u^{\ast}_1(t),\, u^{\ast}_{2}(t),\,\ldots,\,u^{\ast}_{j-1}(t),\, u_{j}(t),\, u^{\ast}_{j+1}(t),\,\ldots,\, u^{\ast}_N(t)\bigr)\in U_{\mathcal{K}},   \label{EQ4}
\end{align}
where $u^{\ast}_i(t) = \mathcal{K}_i^{\ast}(t) x^j(t)$ for $i \in \mathcal{N}_{\neg j} \triangleq \mathcal{N} \backslash \{j\}$ and $j \in \mathcal{N}$. 

Furthermore, we may require that the control input for the $j$th\,-\,channel to be $u_j(t)=\bigl(\mathcal{K}_jx^j\bigr)(t) \in U_j$ and with these linear feedback operators
\begin{align*}
\bigl(\mathcal{K}_j, \mathcal{K}_{\neg j}^{\ast}\bigr) \triangleq \bigl(\mathcal{K}_1^{\ast},\,\ldots,\, \mathcal{K}_{j-1}^{\ast},\,\mathcal{K}_{j},\,\mathcal{K}_{j+1}^{\ast},\, \ldots,\,\mathcal{K}_N^{\ast}\bigr) \in \mathscr{K}.
\end{align*}
Then, the unique solution $\phi_j\bigl(t; t_0, x_0, \bigl(\widehat{u_j(t), u^{\ast}_{\neg j}(t)}\bigr)\bigr)$ will take the form
\begin{align}
\phi_j\bigl(t; t_0, x_0, \bigl(\widehat{u_j(t), u^{\ast}_{\neg j}(t)}\bigr)\bigr) = \underbrace{\Phi^{\mathcal{K}_{\neg j}^{\ast}}(t, t_0)\, \Phi^{\mathcal{K}_j}(t, t_0)}_{\triangleq \Phi_t^{(\mathcal{K}_j,\,\mathcal{K}_{\neg j}^{\ast})}} x_0, ~~ \forall t \in [t_0,\,+\infty),  \label{EQ5}
\end{align}
where (with $t_0 = \tau$) 
\begin{align}
 \frac{\partial\Phi^{\mathcal{K}_{\neg j}^{\ast}}(t, \tau)}{\partial t} &= \Bigl(A(t) + \sum\nolimits_{i \in \mathcal{N}_{\neg j}} B_i(t) \mathcal{K}_i^{\ast}(t) \Bigr)\Phi^{\mathcal{K}_{\neg j}^{\ast}}(t, \tau),  \label{EQ6} \\
 \frac{\partial\Phi^{\mathcal{K}_j}(t, \tau)}{\partial t} &= B_j^{\ast}(t)\,\Phi^{\mathcal{K}_j}(t, \tau), \label{EQ7}
\end{align}
with both $\Phi^{\mathcal{K}_{\neg j}^{\ast}}(\tau, \tau)$ and $\Phi^{\mathcal{K}_j}(\tau, \tau)$ are identity matrices, and $B^{\ast}(t)$ is given by
\begin{align}
 B_j^{\ast}(t) = \Bigl(\Phi^{\mathcal{K}_{\neg j}^{\ast}}(t, \tau)\Bigr)^{-1} B_j(t)\mathcal{K}_j(t) \Phi^{\mathcal{K}_{\neg j}^{\ast}}(t, \tau),  \label{EQ8}
 \end{align}
for each $j \in \mathcal{N}$ (e.g., see \cite{Che62} for such a decomposition that arises in differential equations).

In the following, we assume that $X$ is a topological Hausdorff space and $\mathscr{A}$ is a $\sigma$\,-\,algebra of Borel set, i.e., the smallest $\sigma$\,-\,algebra which contains all open, and thus closed, subsets of $X$. With this, for any $t \ge 0$ (assuming that $t_0=0)$ and $(\mathcal{K}_j,\,\mathcal{K}_{\neg j}^{\ast})\in \mathscr{K}$, we can consider a family of continuous measurable mappings (or transformations) $\Bigl\{\Phi_t^{(\mathcal{K}_j,\,\mathcal{K}_{\neg j}^{\ast})} \Bigr\}_{t \ge 0}$
on $X$ satisfying
\begin{align}
 (\mathbb{R}_{+} \times X) \ni (t,\,x) \mapsto \Phi_t^{(\mathcal{K}_j,\,\mathcal{K}_{\neg j}^{\ast})} x \in X.  \label{EQ9} 
\end{align}

\begin{remark}
Note that Equation~\eqref{EQ5} (which corresponds to the $j$th\,-\,subsystem) appeared to be tractable for the game-theoretic formalism of Subsection~\ref{S3(1)}, when any computational effort is required with respect to $\mathcal{K}_j$ for $j \in \mathcal{N}$, while the others $\mathcal{K}_{\neg j}^{\ast}$ remain fixed (cf. Definition~\ref{DFN2} as well as Footnote~\ref{FN5}).
\end{remark}

\subsection{Entropy of equilibrium measures} \label{S2(2)}
In the following, let $\mathcal{M}_{\Phi_t}\bigl(X\bigr)$ denote the set of all $\Phi_t^{(\mathcal{K}_j,\,\mathcal{K}_{\neg j}^{\ast})}$-\,invariant Borel probability measures on $X$ with respect to $(\mathcal{K}_j,\,\mathcal{K}_{\neg j}^{\ast}) \in \mathscr{K}$. Then, we can introduce the following definition.
\begin{definition}\label{DFN1}
A measure $\nu \in \mathcal{M}_{\Phi_t}\bigl(X\bigr)$ is called invariant under the family of measurable transformations $\Bigl\{\Phi_t^{(\mathcal{K}_j,\,\mathcal{K}_{\neg j}^{\ast})} \Bigr\}_{t \ge 0}$ if
\begin{align}
 \nu\Bigl(\Phi_t^{(\mathcal{K}_j,\,\mathcal{K}_{\neg j}^{\ast})} \bigr)^{-1}(A) = \nu\bigl(A\bigr), \quad \forall t \ge 0,  \label{EQ10}
\end{align}
for any Borel set $A \in \mathscr{A}$.
\end{definition}
\begin{remark}
Note that, for each fixed $t \ge 0$ and $(\mathcal{K}_j,\,\mathcal{K}_{\neg j}^{\ast}) \in \mathscr{K}$, the transformation $\Phi_t^{(\mathcal{K}_j,\,\mathcal{K}_{\neg j}^{\ast})}$ is measurable, i.e., for any Borel set $A \in \mathscr{A}$, the set $\Bigl(\Phi_t^{(\mathcal{K}_j,\,\mathcal{K}_{\neg j}^{\ast})}\Bigr)^{-1}(A) \in \mathscr{A}$ is measurable, where $\Bigl(\Phi_t^{(\mathcal{K}_j,\,\mathcal{K}_{\neg j}^{\ast})}\Bigr)^{-1}(A)$ denotes the set of all $x$ such that $\Phi_t^{(\mathcal{K}_j,\,\mathcal{K}_{\neg j}^{\ast})} x \in A$ (e.g., see \cite{LasPia77} for invariant measures on topological spaces; see also \cite{Sin59} or \cite{Kol58} on the notion of entropy for dynamical systems).
\end{remark}

Next, we have the following proposition that establishes the existence of an equilibrium state for a given potential function.
\begin{proposition}\label{PR1} {\em (Cf. Proposition~3.1 in \cite{Kif90})} 
The measure $\mu \in \mathcal{M}_{\Phi_t}\bigl(X\bigr)$ is called an equilibrium state for the transformation $\Phi_t^{(\mathcal{K}_j, \mathcal{K}_{\neg j}^{\ast}})$, for each fixed $t \ge 0$, and a potential function $\varphi \in C(X, \mathbb{R})$ if the following holds\footnote{$\varphi \in C(X, \mathbb{R})$ is measurable and bounded from above and below.}
\begin{align}
 P_{\varphi}\Bigl(\Phi_t^{(\mathcal{K}_j,\,\mathcal{K}_{\neg j}^{\ast})} \Bigr) &= \sup_{\nu \in \mathcal{M}_{\Phi_t}\bigl(X\bigr)} \left \{h_{\nu}\Bigl(\Phi_t^{(\mathcal{K}_j,\,\mathcal{K}_{\neg j}^{\ast})}\Bigr) + \int_{X} \varphi d\nu  \right\}, \notag \\
 &= h_{\mu}\Bigl(\Phi_t^{(\mathcal{K}_j,\,\mathcal{K}_{\neg j}^{\ast})}\Bigr) + \int_{X} \varphi d\mu, \label{EQ11}
\end{align}
for all $j \in \mathcal{N}$, where $h_{\mu}$ is the measure-theoretic entropy of $\Phi_t^{(\mathcal{K}_j,\,\mathcal{K}_{\neg j}^{\ast})}$.
\end{proposition}

\begin{remark}
We remark that, for each fixed $t \ge 0$ and $(\mathcal{K}_j,\,\mathcal{K}_{\neg j}^{\ast}) \in \mathscr{K}$, the measure-theoretic entropy 
\begin{align*}
\mathscr{K} \times \mathcal{M}_{\Phi_t}\bigl(X\bigr) \ni \bigl((\mathcal{K}_j,\,\mathcal{K}_{\neg j}^{\ast}), \nu \bigr)  \mapsto h_{\nu}\Bigl(\Phi_t^{(\mathcal{K}_j,\,\mathcal{K}_{\neg j}^{\ast})}\Bigr) \in \mathbb{R} \cup \{\infty\},
\end{align*}
is upper semicontinuous (see also \cite[Lemma~2.3]{BruK98}). Moreover, if the potential $\varphi \in C(X, \mathbb{R})$ is upper semicontinuous, then the free energy functional in Equation~\eqref{EQ11} is upper semicontinuous too. Note that the measure space $\mathcal{M}_{\Phi_t}\bigl(X\bigr)$ is compact in the weak topology, this further ensures the existence of equilibrium states with respect to $\varphi \in C(X, \mathbb{R})$.
\end{remark}

\section{Main Results} \label{S3}
\subsection{Game-theoretic formalism} \label{S3(1)}
In the following, we specify a game in a strategic form over an infinite-horizon -- where, in the course of the game, each feedback operator generates automatically a feedback control in response to the action of other feedback operators via the system state $x(t)$ for $t \in [t_0,\,+\infty)$. For example, the $j$th\,-\,feedback operator can generate a feedback control $u_j(t)=\bigl(\mathcal{K}_jx^j\bigr)(t)$ in response to the actions of other feedback operators $u^{\ast}_i(t)= \bigl(\mathcal{K}_i^{\ast} x^j\bigr)(t)$ for $i \in \mathcal{N}_{\neg j}$, where $\bigl(\widehat{u_j(t), u^{\ast}_{\neg j}(t)}\bigr) \in  U_{\mathcal{K}}$, and, similarly, any number of feedback operators can decide on to play feedback strategies simultaneously. Hence, for such a game to have stable equilibrium solution (which is also robust to small perturbations in the system or strategies played by others), then each feedback operator is required to respond optimally (in some sense of best-response correspondence) to the others strategies.

To this end, for a given potential function $\varphi \in C(X, \mathbb{R})$, it will be useful to consider the following criterion function (i.e., the free energy of the dynamical system)\footnote{Such a criterion function (which is related with the sum of the entropy and the integral over a potential) will encode changes, when the set of feedback operators decides simultaneously to change their strategies toward attaining the game-theoretic equilibrium (cf. Rosenthal \cite{Ros73} although in a different way than in the present paper).}

\begin{align}
\mathscr{K} \times \mathcal{M}_{\Phi_t}\bigl(X\bigr) \ni \bigl((\mathcal{K}_j,\,\mathcal{K}_{\neg j}^{\ast}), \nu \bigr) \mapsto P_{\varphi}\Bigl(\Phi_t^{(\mathcal{K}_j,\,\mathcal{K}_{\neg j}^{\ast})}\Bigr) \in \mathbb{R} \cup \{\infty\}, ~~ \forall t\ge 0, ~~ \forall j \in \mathcal{N}, \label{EQ12} 
\end{align}
over the class of linear feedback operators $\mathscr{K}$ (or the class of admissible control functions $U_{\mathcal{K}}$). Note that, under the game-theoretic formalism, if there exist game-theoretic equilibrium feedback operators $\bigl(\mathcal{K}_1^{\ast},\, \mathcal{K}_2^{\ast},\, \ldots,\, \mathcal{K}_N^{\ast}\bigr) \in \mathscr{K}$ that strategically interact using the current state-information of the system. Then, for a given potential $\varphi \in C(X, \mathbb{R})$, the problem that arises from attempting to describe the equilibrium states for the family of measurable transformations $\Bigl\{\Phi_t^{(\mathcal{K}_j,\,\mathcal{K}_{\neg j}^{\ast})} \Bigr\}_{t \ge 0}$, $(\mathcal{K}_j,\,\mathcal{K}_{\neg j}^{\ast}) \in \mathscr{K}$, will also include determining all ergodic measures that maximize entropy for the multi-channel dynamical system, while the latter is composed with these equilibrium feedback operators. 

Hence, more formally, we have the following definition for the equilibrium feedback operators $(\mathcal{K}_1^{\ast},\, \mathcal{K}_2^{\ast}, \, \ldots, \, \mathcal{K}_N^{\ast}) \in \mathscr{K}$.

\begin{definition}\label{DFN2}
We shall say that the feedback operators $(\mathcal{K}_1^{\ast},\, \mathcal{K}_2^{\ast}, \, \ldots, \, \mathcal{K}_N^{\ast}) \in \mathscr{K}$ are the game-theoretic equilibrium feedback operators, if they produce control responses given by
\begin{align}
 u_j^{\ast}(t) = \bigl(\mathcal{K}_j^{\ast}x\bigr)(t) \in U_j, \quad \forall j \in \mathcal{N}, \label{EQ13}
\end{align}
 for $t \in [0, \infty]$ and satisfy the following conditions
\begin{align}
 P_{\varphi}\Bigl(\Phi_t^{(\mathcal{K}_j,\,\mathcal{K}_{\neg j}^{\ast})}\Bigr)  \le \underbrace{\sup_{(\mathcal{K}_j,\,\mathcal{K}_{\neg j}^{\ast}) \in \mathscr{K}} \, \sup_{\nu \in \mathcal{M}_{\Phi_t}\bigl(X\bigr)} \left \{ h_{\nu}\Bigl(\Phi_t^{(\mathcal{K}_j,\,\mathcal{K}_{\neg j}^{\ast})}\Bigr) + \int_{X} \varphi d\nu  \right\}}_{\triangleq P_{\varphi}\Bigl(\Phi_t^{(\mathcal{K}_j^{\ast},\,\mathcal{K}_{\neg j}^{\ast})}\Bigr)}, ~~ \forall t \ge 0. \label{EQ14} 
\end{align}
where the outer supremum in the above equation is determined with respect to $\mathcal{K}_j$ for each $j \in \mathcal{N}$; while the others $\mathcal{K}_{\neg j}^{\ast}$ remain fixed.\footnote{\label{FN5}Note that, in Equation~\eqref{EQ14}, the computational effort is performed with respect to $\mathcal{K}_j$ for each $j \in \mathcal{N}$, while the others $\mathcal{K}_{\neg j}^{\ast}$ remain fixed. Then, the following is more appropriate (in the sense of set inclusion)
\begin{align*}
 \mathcal{K}_j^{\ast} \in \underset{(\mathcal{K}_j,\,\mathcal{K}_{\neg j}^{\ast}) \in \mathscr{K}}{\operatorname{arg\,sup}} \left\{ \sup_{\nu \in \mathcal{M}_{\Phi_t}\bigl(X\bigr)} \left \{ h_{\nu}\Bigl(\Phi_t^{(\mathcal{K}_j,\,\mathcal{K}_{\neg j}^{\ast})}\Bigr) + \int_{X} \varphi d\nu \right\} \right\}, \quad \forall t \ge 0,
\end{align*}
(e.g., see Aubin \cite{Aub93} for such notions arising in nonlinear analysis).}
\end{definition}

\begin{remark} \label{RM4}
Note that the inner supremum is over all $\Phi_t^{(\mathcal{K}_j,\,\mathcal{K}_{\neg j}^{\ast})}$-\,invariant Borel probability measures on $X$. Further, if there exist game-theoretic  equilibrium feedback operators, then this implies that the set of equilibrium states for $X$, i.e.,
\begin{align*}
 \mathcal{E}_{\varphi}\bigl(X\bigr) = \bigcap_{t \ge 0} \biggl \{\mu \in \mathcal{M}_{\Phi_t} \Bigl \vert P_{\varphi}\Bigl(\Phi_t^{(\mathcal{K}_j^{\ast},\,\mathcal{K}_{\neg j}^{\ast})} \Bigr) = h_{\mu}\Bigl(\Phi_t^{(\mathcal{K}_j^{\ast},\,\mathcal{K}_{\neg j}^{\ast})}\Bigr) + \int_{X} \varphi d\mu \biggr\},
 \end{align*}
is non-empty with respect to $\Bigl\{\Phi_t^{(\mathcal{K}_j^{\ast},\,\mathcal{K}_{\neg j}^{\ast})} \Bigr\}_{t \ge 0}$ and $\varphi \in C(X, \mathbb{R})$.
\end{remark}
\begin{remark}
We remark that, for $\varphi = 0$, the entropy that corresponds to $P_{0} \bigl(\Phi_t^{(\mathcal{K}_j^{\ast},\,\mathcal{K}_{\neg j}^{\ast})}\bigr)$ is equal to the topological entropy $h_{\rm top}\bigl(\Phi_t^{(\mathcal{K}_j^{\ast},\,\mathcal{K}_{\neg j}^{\ast})}\bigr)$ (e.g., see Adler \etal\, \cite{AdlKM65} or Walters \cite{Wal82}).
\end{remark}

Then, we formally state the main objective of this paper.

{\bf Problem}: For a given potential function $\varphi \in C(X, \mathbb{R})$, provide a sufficient condition on the existence of game-theoretic equilibrium feedback operators $\bigl(\mathcal{K}_1^{\ast},\, \mathcal{K}_2^{\ast},\, \ldots,\, \mathcal{K}_N^{\ast}\bigr) \in \mathscr{K}$ that guarantee maximum entropy of equilibrium states for the measurable transformation $\Phi_t^{(\mathcal{K}_j^{\ast},\,\mathcal{K}_{\neg j}^{\ast})}$ for all $t \ge 0$.

\subsection{Existence of game-theoretic equilibrium feedback operators}\label{S3(2)}
In the following, we provide a sufficient condition for the existence of an equilibrium state that is associated with an equilibrium measure (i.e., a common limit point) for the family of operators $\Bigl\{\mathcal{L}_{(\varphi,\,t)}^{(\mathcal{K}_j,\,\mathcal{K}_{\neg j}^{\ast})}\Bigr\}_{t \ge 0}$, for each $j \in \mathcal{N}$, i.e.,
\begin{align}
 \left\{\left(\mathcal{L}_{(\varphi,\,t)}^{(\mathcal{K}_j,\,\mathcal{K}_{\neg j}^{\ast})} \varrho\right)(x)\right\}_{t \ge 0} =  \left\{\int_{\Bigl(\Phi_t^{(\mathcal{K}_j,\,\mathcal{K}_{\neg j}^{\ast})} \bigr)^{-1}(x)} \exp\bigl(\varphi(y)\bigr)\varrho(y)d\nu\right\}_{t \ge 0},\notag \\
 \forall \nu \in \mathcal{M}_{\Phi_t}\bigl(X\bigr), \label{EQ15}
\end{align}
where $\varrho(x)$ is any measurable function from the space of all possible real-valued measurable functions on $X$ satisfying $\int_{X}\vert \varrho(x)\vert d \nu < \infty$.
\begin{proposition}\label{PR2}
Suppose that the family of linear operators $\Bigl\{\mathcal{L}_{(\varphi,\,t)}^{(\mathcal{K}_j,\,\mathcal{K}_{\neg j}^{\ast})}\Bigr\}_{t \ge 0}$ in Equation~\eqref{EQ15} satisfies the following
\begin{align}
 \bigcap_{j \in \mathcal{N}} \bigcap_{t \ge 0} \left(\mathcal{L}_{(\varphi,\,t)}^{(\mathcal{K}_j,\,\mathcal{K}_{\neg j}^{\ast})} \varrho\right)(x) \neq \varnothing, \quad \forall (\mathcal{K}_j,\,\mathcal{K}_{\neg j}^{\ast}) \in \mathscr{K}, \label{EQ16}
\end{align}
for a given potential $\varphi \in C(X, \mathbb{R})$. 

Then, there exist game-theoretic equilibrium feedback operators $(\mathcal{K}_1^{\ast},\, \mathcal{K}_2^{\ast}, \, \ldots, \, \mathcal{K}_N^{\ast}) \in \mathscr{K}$, if the following hold
\begin{align}
 P_{\varphi}\Bigl(\Phi_t^{(\mathcal{K}_j,\,\mathcal{K}_{\neg j}^{\ast})}\Bigr)  \le P_{\varphi}\Bigl(\Phi_t^{(\mathcal{K}_j^{\ast},\,\mathcal{K}_{\neg j}^{\ast})}\Bigr), \quad \forall t \ge 0. \label{EQ17}
\end{align}
and
\begin{align}
 \mathcal{L}_{(\varphi,\,t)}^{(\mathcal{K}_j^{\ast},\,\mathcal{K}_{\neg j}^{\ast})} \varrho \rightarrow  \mu  \quad  \text{as} \quad t \rightarrow \infty. \label{EQ18} 
\end{align}
where $\mu \in \mathcal{E}_{\varphi}\bigl(X\bigr) \subset \mathcal{M}_{\Phi_t}\bigl(X\bigr)$ is an equilibrium state that satisfies
\begin{align*}
 P_{\varphi}\Bigl(\Phi_t^{(\mathcal{K}_j^{\ast},\,\mathcal{K}_{\neg j}^{\ast})}\Bigr) = h_{\mu}\Bigl(\Phi_t^{(\mathcal{K}_j^{\ast},\,\mathcal{K}_{\neg j}^{\ast})}\Bigr) + \int_{X} \varphi d\mu, \quad \forall t \ge 0.
\end{align*}
\end{proposition}
\begin{proof}
Note that if $\mathcal{E}_{\varphi}\bigl(X\bigr)$ is non-empty, then Equation~\eqref{EQ16} holds true for all $(\mathcal{K}_j,\,\mathcal{K}_{\neg j}^{\ast}) \in \mathscr{K}$. Let
\begin{align*}
 \mu \in \underbrace{\bigcap_{j \in \mathcal{N}} \bigcap_{t \ge 0} \overline{\left \{ \nu \in \mathcal{M}_{\Phi_t}\bigl(X\bigr) \,\Bigl \vert\, h_{\nu}\Bigl(\Phi_t^{(\mathcal{K}_j,\,\mathcal{K}_{\neg j}^{\ast})}\Bigr) + \int_{X} \varphi d\nu > P_{\varphi}\Bigl(\Phi_t^{(\mathcal{K}_j,\,\mathcal{K}_{\neg j}^{\ast})}\Bigr) - \frac{1}{n} \right\}}}_{\triangleq\mathcal{I}_{\varphi}\bigl(X\bigr)}. 
\end{align*}
and noticing that $\mathcal{E}_{\varphi}\bigl(X\bigr) \subset  \mathcal{I}_{\varphi}\bigl(X\bigr)$ (cf. Remark~\ref{RM4}), then there exists a family of invariant measures $\bigl\{\nu_n\bigr\}_{n \ge 1}$, $\nu_n \in \mathcal{M}_{\Phi_t}\bigl(X\bigr)$, such that
\begin{align*}
\nu_n \rightarrow \mu \quad  \text{as} \quad n \rightarrow \infty,
\end{align*}
and for all $t \ge 0$
\begin{align*}
\sup_{(\mathcal{K}_j,\,\mathcal{K}_{\neg j}^{\ast}) \in \mathscr{K}} \left \{h_{\nu_n}\Bigl(\Phi_t^{(\mathcal{K}_j,\,\mathcal{K}_{\neg j}^{\ast})}\Bigr) + \int_{X} \varphi d\nu_n\right\} \rightarrow \sup_{(\mathcal{K}_j,\,\mathcal{K}_{\neg j}^{\ast}) \in \mathscr{K}}\, \underbrace{P_{\varphi}\Bigl(\Phi_t^{(\mathcal{K}_j,\,\mathcal{K}_{\neg j}^{\ast})}\Bigr)}_{\triangleq P_{\varphi}\Bigl(\Phi_t^{(\mathcal{K}_j^{\ast},\,\mathcal{K}_{\neg j}^{\ast})}\Bigr)},
\end{align*}
which corresponds to the equilibrium state, when the feedback operators attain the game-theoretic equilibrium, i.e., $(\mathcal{K}_j^{\ast},\,\mathcal{K}_{\neg j}^{\ast}) \in \mathscr{K}$.\footnote{Here, we implicitly assume that there exists at least one common limit point of equilibrium state.}

Then, we have
\begin{align*}
  P_{\varphi + \widetilde{\varphi}} \Bigl(\Phi_t^{(\mathcal{K}_j^{\ast},\,\mathcal{K}_{\neg j}^{\ast})}\Bigr) - P_{\varphi} \Bigl(\Phi_t^{(\mathcal{K}_j^{\ast},\,\mathcal{K}_{\neg j}^{\ast})}\Bigr) &\ge h_{\nu_n}\Bigl(\Phi_t^{(\mathcal{K}_j^{\ast},\,\mathcal{K}_{\neg j}^{\ast})}\Bigr) + \int_{X} (\varphi + \widetilde{\varphi}) d\nu_n \notag \\
   &- P_{\varphi} \Bigl(\Phi_t^{(\mathcal{K}_j^{\ast},\,\mathcal{K}_{\neg j}^{\ast})}\Bigr) \rightarrow \int_{X} \widetilde{\varphi} d\mu, \,\, \forall \widetilde{\varphi} \in C(X, \mathbb{R}). 
\end{align*}
as $n \rightarrow \infty$. Hence, we have $\mu \in \mathcal{E}_{\varphi}\bigl(X\bigr)$.

Now suppose that there exist $\mu \in \mathcal{M}_{\Phi_t}\bigl(X\bigr) \backslash \mathcal{I}_{\varphi}\bigl(X\bigr)$ and $(\mathcal{K}_j^{\ast},\,\mathcal{K}_{\neg j}^{\ast}) \in \mathscr{K}$ such that Equation~\eqref{EQ17} holds. Notice that $\mathcal{I}_{\varphi}\bigl(X\bigr)$ is compact in the weak topology. Then, using the separation theorem (e.g., see \cite[pp 417]{DunSch58}), we have
\begin{align*}
   \int_{X} \widetilde{\varphi} d\mu > \sup \left\{\int_{X} \widetilde{\varphi} d\nu \, \Bigl\vert \,\nu \in \mathcal{I}_{\varphi}\bigl(X\bigr) \right\},
\end{align*}
for some $\widetilde{\varphi} \in C(X, \mathbb{R})$ which is bounded from above and below. 

Next, for each $n \ge 1$, if we choose $\nu_n$ such that
\begin{align*}
 h_{\nu_n}\Bigl(\Phi_t^{(\mathcal{K}_j^{\ast},\,\mathcal{K}_{\neg j}^{\ast})}\Bigr) + \int_{X} \bigl(\varphi + \frac{\widetilde{\varphi}}{n}\bigr) d\nu_n  > P_{\varphi + \frac{\widetilde{\varphi}}{n}} \Bigl(\Phi_t^{(\mathcal{K}_j^{\ast},\,\mathcal{K}_{\neg j}^{\ast})}\Bigr) - \frac{1}{n^2}.
\end{align*}
Then, we have
\begin{align*}
& \int_{X} \widetilde{\varphi} d\mu = n \int_{X} \frac{\widetilde{\varphi}}{n} d\mu \le n \left [P_{\varphi + \frac{\widetilde{\varphi}}{n}} \Bigl(\Phi_t^{(\mathcal{K}_j^{\ast},\,\mathcal{K}_{\neg j}^{\ast})}\Bigr) - P_{\varphi} \Bigl(\Phi_t^{(\mathcal{K}_j^{\ast},\,\mathcal{K}_{\neg j}^{\ast})}\Bigr)\right] \\
 & \quad\le n \left [h_{\nu_n}\Bigl(\Phi_t^{(\mathcal{K}_j^{\ast},\,\mathcal{K}_{\neg j}^{\ast})}\Bigr) + \int_{X} \bigl(\varphi + \frac{\widetilde{\varphi}}{n}\bigr) d\nu_n - h_{\nu_n}\Bigl(\Phi_t^{(\mathcal{K}_j^{\ast},\,\mathcal{K}_{\neg j}^{\ast})}\Bigr) -  \int_{X} \varphi d\nu_n + \frac{1}{n^2}\right] \\
 & \quad =  \int_{X}  \widetilde{\varphi} d\nu_n + \frac{1}{n}.
\end{align*}
From the above equation, if $\tilde{\mu}$ is a common limit point of $\bigl\{\nu_n\bigr\}_{n \ge 1}$, then $\int_{X} \widetilde{\varphi} d\mu \le \int_{X} \widetilde{\varphi} d\tilde{\mu}$. But, this is a contradiction, since $\tilde{\mu} \in \mathcal{I}_{\varphi}\bigl(X\bigr)$ follows from
\begin{align*}
 h_{\nu_n}\Bigl(\Phi_t^{(\mathcal{K}_j^{\ast},\,\mathcal{K}_{\neg j}^{\ast})}\Bigr) + \int_{X} \varphi d\nu_n & > P_{\varphi + \frac{\widetilde{\varphi}}{n}} \Bigl(\Phi_t^{(\mathcal{K}_j^{\ast},\,\mathcal{K}_{\neg j}^{\ast})}\Bigr) - \int_{X} \frac{\widetilde{\varphi}}{n} d\nu_n -  \frac{1}{n^2}.
\end{align*}
Hence, we easily see that $\mu \in \mathcal{E}_{\varphi}\bigl(X\bigr)$ is the equilibrium measure, which corresponds to the game-theoretic equilibrium feedback operators $(\mathcal{K}_j^{\ast},\,\mathcal{K}_{\neg j}^{\ast}) \in \mathscr{K}$ (i.e., the statement in Equation~\eqref{EQ17} hold true for all $t \ge 0$). On the other hand, if $\mu \in \mathcal{E}_{\varphi}\bigl(X\bigr)$ (i.e., the equilibrium measure), then there exists a measurable function $\varrho(x)$ (from a suitable function space on $X$) such that
\begin{align*}
  \mathcal{L}_{(\varphi,\,t)}^{(\mathcal{K}_j^{\ast},\,\mathcal{K}_{\neg j}^{\ast})} \varrho \in \mathcal{E}_{\varphi}\bigl(X\bigr), \quad \forall t \ge 0,
\end{align*}
and
\begin{align*}
 \lim_{t \rightarrow \infty} \, \mathcal{L}_{(\varphi,\,t)}^{(\mathcal{K}_j^{\ast},\,\mathcal{K}_{\neg j}^{\ast})} \varrho \rightarrow \mu.
\end{align*}
This completes the proof. \qed
\end{proof}

\begin{remark}
The last two equations state that, for any $t \ge 0$ and starting from any initial state, the best response path will always lead to the game-theoretic equilibrium feedback operators.
\end{remark}

\subsection{Resilient behavior of game-theoretic equilibrium feedback operators} \label{S3(3)}
In this subsection, we consider the following system with a small random perturbation term
\begin{align} 
 dZ^{\epsilon}(t) = \Bigl(A(t) + \sum\nolimits_{j \in \mathcal{N}} B_j(t) \mathcal{K}_j^{\ast}(t) \Bigr)Z^{\epsilon}(t) dt + \sqrt{\epsilon} \,\sigma(Z^{\epsilon}(t)) dW(t), \,\, Z^{\epsilon}(0)=x_0, \label{EQ19} 
 \end{align}
where $\sigma(Z^{\epsilon}(t)) \in \mathbb{R}^{d \times d}$ is a diffusion term, $W(t)$ is a $d$-dimensional Wiener process and $\epsilon$ is a small positive number, which represents the level of perturbation in the system. Note that we assume here there exist equilibrium feedback operators $\bigl(\mathcal{K}_1^{\ast}, \mathcal{K}_2^{\ast}, \ldots, \mathcal{K}_N^{\ast} \bigr) \in \mathscr{K}$, when $\epsilon = 0$ (which corresponds to the unperturbed multi-channel dynamical system). Then, we investigate, as $\epsilon \rightarrow 0$, the asymptotic stability behavior of an equilibrium measure for the family of linear operators $\Bigl\{\mathcal{L}_{(\varphi^{\epsilon},\,t)}^{(\mathcal{K}_j^{\ast},\,\mathcal{K}_{\neg j}^{\ast})} \Bigr\}_{t \ge 0}$ that corresponds to the multi-channel dynamical system with a small random perturbation.

\begin{remark}
We remark that such a solution for Equation~\eqref{EQ19} is assumed to have continuous sample paths with probability one (e.g., see Kunita \cite{Kun90}).
\end{remark}
Next, we present the following proposition, which is concerned with the resilient behavior of the equilibrium feedback operators -- when there is a small random perturbation in the system. 
\begin{proposition}\label{PR3}
For $\epsilon > 0$ and any measurable function $\varrho(x)$, if the family of linear operators $\Bigl\{\mathcal{L}_{(\varphi^{\epsilon},\,t)}^{(\mathcal{K}_j^{\ast},\,\mathcal{K}_{\neg j}^{\ast})} \Bigr\}_{t \ge 0}$ satisfies the following relative entropy condition with respect to the unperturbed dynamical system
\begin{align}
\mathcal{R}\Bigr(\mathcal{L}_{(\varphi^{\epsilon},\,t)}^{(\mathcal{K}_j^{\ast},\,\mathcal{K}_{\neg j}^{\ast})}\varrho \,\Bigl \Vert \,\mathcal{L}_{(\varphi,\,t)}^{(\mathcal{K}_j^{\ast},\,\mathcal{K}_{\neg j}^{\ast})}\varrho \Bigl) &= \int_{X} \log \left(\frac{d\,\mathcal{L}_{(\varphi^{\epsilon},\,t)}^{(\mathcal{K}_j^{\ast},\,\mathcal{K}_{\neg j}^{\ast})}\varrho} {d\,\mathcal{L}_{(\varphi,\,t)}^{(\mathcal{K}_j^{\ast},\,\mathcal{K}_{\neg j}^{\ast})}\varrho} \right) d\,\mathcal{L}_{(\varphi^{\epsilon},\,t)}^{(\mathcal{K}_j^{\ast},\,\mathcal{K}_{\neg j}^{\ast})}\varrho \notag\\
&\le\, \rho_{(\epsilon,\,t)}^{(\mathcal{K}_j^{\ast},\,\mathcal{K}_{\neg j}^{\ast})}, \quad \mathcal{L}_{(\varphi^{\epsilon},\,t)}^{(\mathcal{K}_j^{\ast},\,\mathcal{K}_{\neg j}^{\ast})}\varrho \ll\mathcal{L}_{(\varphi,\,t)}^{(\mathcal{K}_j^{\ast},\,\mathcal{K}_{\neg j}^{\ast})}\varrho , \,\, \forall t \ge 0, \label{EQ20} 
\end{align}
where $\rho_{(\epsilon,\,t)}^{(\mathcal{K}_j^{\ast},\,\mathcal{K}_{\neg j}^{\ast})}$ is the entropy gap, which is a positive quantity (that depends on $\epsilon$ and also tends to zero as $\epsilon \rightarrow 0$). 

Then, the equilibrium feedback operators $\bigl(\mathcal{K}_1^{\ast}, \mathcal{K}_2^{\ast}, \ldots, \mathcal{K}_N^{\ast} \bigr) \in \mathscr{K}$ exhibit a resilient behavior.\footnote{Note that $\varrho(x)$ is assumed to satisfy $\int_{X} \varrho(x)d\nu < \infty$ (cf. Equation~\eqref{EQ15}).}
\end{proposition}
\begin{proof}
Suppose that the set of feedback operators $\bigl(\mathcal{K}_1^{\ast}, \mathcal{K}_2^{\ast}, \ldots, \mathcal{K}_N^{\ast} \bigr) \in \mathscr{K}$ satisfies Proposition~\ref{PR2}. Then, for each $\epsilon \ge 0$ that is sufficiently small, the following lower bound condition holds\footnote{We remark that such a lower bound on the relative entropy between two measurable functions from the $L^1$ space was originally proved by Csisz\'{a}r \cite{Csi67}.}
\begin{align*}
  \frac{1}{2}\Bigl \Vert d\,\mathcal{L}_{(\varphi^{\epsilon},\,t)}^{(\mathcal{K}_j^{\ast},\,\mathcal{K}_{\neg j}^{\ast})}\varrho - d\,\mathcal{L}_{(\varphi,\,t)}^{(\mathcal{K}_j^{\ast},\,\mathcal{K}_{\neg j}^{\ast})}\varrho \Bigr\Vert_{L^1}^2 \le \mathcal{R}\Bigr(\mathcal{L}_{(\varphi^{\epsilon},\,t)}^{(\mathcal{K}_j^{\ast},\,\mathcal{K}_{\neg j}^{\ast})}\varrho \,\Bigl \Vert \,\mathcal{L}_{(\varphi,\,t)}^{(\mathcal{K}_j^{\ast},\,\mathcal{K}_{\neg j}^{\ast})}\varrho \Bigl), ~~ \forall t \ge 0.
 \end{align*}
On the other hand, suppose that  $\operatorname{supp} d\,\mathcal{L}_{(\varphi^{\epsilon},\,t)}^{(\mathcal{K}_j^{\ast},\,\mathcal{K}_{\neg j}^{\ast})}\varrho(x) \subset \operatorname{supp}d\,\mathcal{L}_t^{(\mathcal{K}_j^{\ast},\,\mathcal{K}_{\neg j}^{\ast})}\varrho(x)$, $\forall t \ge 0$ and $\forall x \in X$. Then, we have
\begin{align*}
& \mathcal{R}\Bigr(\mathcal{L}_{(\varphi^{\epsilon},\,t)}^{(\mathcal{K}_j^{\ast},\,\mathcal{K}_{\neg j}^{\ast})}\varrho \,\Bigl \Vert \,\mathcal{L}_{(\varphi,\,t)}^{(\mathcal{K}_j^{\ast},\,\mathcal{K}_{\neg j}^{\ast})}\varrho \Bigl)  \\
 &\quad \quad = \int_{X} \log \left(\frac{d\,\mathcal{L}_{(\varphi^{\epsilon},\,t)}^{(\mathcal{K}_j^{\ast},\,\mathcal{K}_{\neg j}^{\ast})}\varrho} {d\,\mathcal{L}_{(\varphi,\,t)}^{(\mathcal{K}_j^{\ast},\,\mathcal{K}_{\neg j}^{\ast})}\varrho} \right) d\,\mathcal{L}_{(\varphi^{\epsilon},\,t)}^{(\mathcal{K}_j^{\ast},\,\mathcal{K}_{\neg j}^{\ast})}\varrho, \\
                                                      &\quad \quad  \le \int_{X} \left(\frac{d\,\mathcal{L}_{(\varphi^{\epsilon},\,t)}^{(\mathcal{K}_j^{\ast},\,\mathcal{K}_{\neg j}^{\ast})}\varrho} {d\,\mathcal{L}_{(\varphi,\,t)}^{(\mathcal{K}_j^{\ast},\,\mathcal{K}_{\neg j}^{\ast})}\varrho} \right) \left (d\,\mathcal{L}_{(\varphi^{\epsilon},\,t)}^{(\mathcal{K}_j^{\ast},\,\mathcal{K}_{\neg j}^{\ast})}\varrho - d\,\mathcal{L}_{(\varphi,\,t)}^{(\mathcal{K}_j^{\ast},\,\mathcal{K}_{\neg j}^{\ast})}\varrho\right) d\,\mathcal{L}_{(\varphi^{\epsilon},\,t)}^{(\mathcal{K}_j^{\ast},\,\mathcal{K}_{\neg j}^{\ast})}\varrho, \\
                                                     &\quad \quad  \le \int_{X} \frac{1} {d\,\mathcal{L}_{(\varphi,\,t)}^{(\mathcal{K}_j^{\ast},\,\mathcal{K}_{\neg j}^{\ast})}\varrho} \left( \left(d\,\mathcal{L}_{(\varphi^{\epsilon},\,t)}^{(\mathcal{K}_j^{\ast},\,\mathcal{K}_{\neg j}^{\ast})}\varrho\right)^2 - \left(d\,\mathcal{L}_{(\varphi,\,t)}^{(\mathcal{K}_j^{\ast},\,\mathcal{K}_{\neg j}^{\ast})}\varrho\right)^2\right) d\,\mathcal{L}_{(\varphi^{\epsilon},\,t)}^{(\mathcal{K}_j^{\ast},\,\mathcal{K}_{\neg j}^{\ast})}\varrho, \\
                                                     &\quad \quad  \le \frac{1}{\mu \bigl(X \bigr)} \int_{X} \left( \left(d\,\mathcal{L}_{(\varphi^{\epsilon},\,t)}^{(\mathcal{K}_j^{\ast},\,\mathcal{K}_{\neg j}^{\ast})}\varrho\right)^2 - \left(d\,\mathcal{L}_{(\varphi,\,t)}^{(\mathcal{K}_j^{\ast},\,\mathcal{K}_{\neg j}^{\ast})}\varrho\right)^2\right) d\,\mathcal{L}_{(\varphi^{\epsilon},\,t)}^{(\mathcal{K}_j^{\ast},\,\mathcal{K}_{\neg j}^{\ast})}\varrho.
\end{align*}
Note that both measurable functions $\mathcal{L}_{(\varphi^{\epsilon},\,t)}^{(\mathcal{K}_j^{\ast},\,\mathcal{K}_{\neg j}^{\ast})}\varrho$ and $\mathcal{L}_{(\varphi,\,t)}^{(\mathcal{K}_j^{\ast},\,\mathcal{K}_{\neg j}^{\ast})}\varrho$ lie in the $L^1$ space for each $t \ge 0$. Hence, the above last expression is bounded from above by a constant (which also depends implicitly on the perturbation level $\epsilon \ge 0$). Further, if we let
\begin{align*}
  \rho_{(\epsilon,\,t)}^{(\mathcal{K}_j^{\ast},\,\mathcal{K}_{\neg j}^{\ast})} = \inf_{t \ge 0} \frac{1}{\mu \bigl(X \bigr)} \int_{X} \left( \left(d\,\mathcal{L}_{(\varphi^{\epsilon},\,t)}^{(\mathcal{K}_j^{\ast},\,\mathcal{K}_{\neg j}^{\ast})}\varrho\right)^2 - \left(d\,\mathcal{L}_{(\varphi,\,t)}^{(\mathcal{K}_j^{\ast},\,\mathcal{K}_{\neg j}^{\ast})}\varrho\right)^2\right) d\,\mathcal{L}_{(\varphi^{\epsilon},\,t)}^{(\mathcal{K}_j^{\ast},\,\mathcal{K}_{\neg j}^{\ast})}\varrho.
\end{align*}
Then, we see that Equation~\eqref{EQ20} holds -- where such a statement evidently establishes the resilient behavior of the equilibrium feedback operators $\bigl(\mathcal{K}_1^{\ast}, \mathcal{K}_2^{\ast}, \ldots, \mathcal{K}_N^{\ast} \bigr) \in \mathscr{K}$, when there is a random perturbation in the system. This completes the proof. \qed
\end{proof}
The above proposition states that the equilibrium feedback operators exhibit a resilient behavior, when the contribution of the random perturbation term, to move away the system from the equilibrium measure $\mu$, is bounded from above for all $t \ge 0$.\footnote{Heuristically, $\rho_{(\epsilon,\,t)}^{(\mathcal{K}_j^{\ast},\,\mathcal{K}_{\neg j}^{\ast})}$, for $t \ge 0$, is the amount of entropy which is attributed due to the presence of random perturbations in the system.}

We also note that the following holds
\begin{align}
\lim_{t \rightarrow \infty} \mathcal{L}_{(\varphi^{\epsilon},\,t)}^{(\mathcal{K}_j^{\ast},\,\mathcal{K}_{\neg j}^{\ast})} \varrho \rightarrow  \mu  \quad  \text{as} \quad \epsilon \rightarrow 0. \label{EQ21} 
\end{align} 
Therefore, such a statement in Equation~\eqref{EQ20} is an immediate consequence of this fact (e.g., see Kifer \cite{Kif88} on the convergence of equilibrium measures for dynamical systems with small random perturbations, in the sense of deterministic limit).

\begin{remark}
Although we have not discussed the limiting behavior, as $\epsilon \rightarrow 0$, in Equation~\eqref{EQ21}, it appears that the {\em theory of large deviations} can be used to estimate the rate at which this family of measures converges to a unique equilibrium measure (e.g., see also \cite{You90}, \cite{Tou09} or \cite{Kif90} for background information).
\end{remark}

\begin{acknowledgements}
This work was supported in part by the National Science Foundation under Grant No. CNS-1035655. The author acknowledges support from the Department of Electrical Engineering, University of Notre Dame.
\end{acknowledgements}


\begin{thebibliography}{99}

\bibitem{AdlKM65}
Adler~R, Konheim~G \& McAndrew~MH (1965)
Topological entropy.
Trans Amer Math Soc 114:309--319

\bibitem{Aub93}
Aubin~J-P (1993)
Optima and equilibria: an introduction to nonlinear analysis.
Springer, Berlin

\bibitem{BefAn13}
Befekadu~GK, Antsaklis~PJ (2013)
Relating maximum entropy, resilient behavior and game-theoretic equilibrium feedback operators in multi-channel systems.
\href{http://arxiv.org/abs/1312.5168}{arXiv:1312.5168} \href{http://arxiv.org/archive/math.CT}{[math.CT]}

\bibitem{BruK98}
Bruin~H \& Keller~G (1998)
Equilibrium states for $s$-unimodal maps.
Ergod Theory Dyn Sys 18:765--789

\bibitem{Che62}
Chen~KT (1962)
Decomposition of differential equations.
Math Ann 146:263--278

\bibitem{Csi67}
Csisz\'{a}r~I (1967)
Information-type measures of difference of probability distributions and indirect observations.
Studia Sci Math Hungar 2:299--318

\bibitem{DunSch58}
Dunford~N, Schwartz~JT (1958)
Linear operators I. 
Interscience, New York

\bibitem{Kel98}
Keller~G (1998)
Equilibrium states in ergodic theory.
Cambridge University Press, Cambridge

\bibitem{Kif88}
Kifer~Y (1988)
Random perturbation of dynamical systems.
Birkh\"{a}user, Boston

\bibitem{Kif90}
Kifer~Y (1990)
Large deviations in dynamical systems and stochastic processes.
Trans Amer Math Soc 321:505--524

\bibitem{Kol58}
Kolmogorov~AN (1958)
A new metric invariant of transient dynamical systems and automorphisms in Lebesgue spaces.
Dokl Akad Nauk SSSR 119(5):861--864

\bibitem{Kun90}
Kunita~H (1990)
Stochastic flows and stochastic differential equations. 
Cambridge University Press, Cambridge

\bibitem{LasPia77}
Lasota~A, Pianigiani~G (1977)
Invariant measures on topological spaces.
Boll Unione Mat Ital 5(15-B):592--603

\bibitem{Nas51}
Nash~J (1951)
Noncooperative games. 
Ann Math 54:286--295

\bibitem{Ros73}
Rosenthal~RW (1973)
A class of games possessing pure strategy Nash equilibria.
Int J Game Theo 2:65--67

\bibitem{Rue78}
Ruelle~D (1978)
Thermodynamic formalism.
Reading, Addison-Wesley, Massachusetts

\bibitem{Sin59}
Sinai~YG (1959) 
On the notion of entropy for a dynamic system. 
Dokl Akad Nauk SSSR 124:768--771

\bibitem{Sin72}
Sinai~YG (1972)
Gibbs measures in ergodic theory.
Usp Mat Nauk 27:21--64

\bibitem{Tou09}
Touchette~H (2009)
The large deviation approach to statistical mechanics.
Phys Rep 478(1-3):1--69

\bibitem{Wal82}
Walters~P (1982)
An introduction to ergodic theory.
Springer, New York

\bibitem{You90}
Young~L-S (1990)
Some large deviation results for dynamical systems.
Trans Amer Math Soc 318:525--543

\end{thebibliography}
\end{document}